\documentclass[11pt,a4paper]{amsart}
\usepackage{amssymb,amsmath,amsthm}
\usepackage{graphicx
}

\newtheorem*{theorem*}{Theorem}

\newtheorem{lemma}{Lemma}

\renewcommand{\geq}{\geqslant}
\renewcommand{\leq}{\leqslant}

\renewcommand{\a}{\mathbf{a}}
\renewcommand{\b}{\mathbf{b}}
\renewcommand{\c}{\mathbf{c}}

\newcommand{\x}{\mathbf{x}}
\newcommand{\y}{\mathbf{y}}
\newcommand{\z}{\mathbf{z}}

\newcommand{\uu}{\mathbf{u}}

\newcommand{\0}{\mathbf{0}}

\newcommand{\R}{\mathbb{R}}

\newcommand{\la}{\langle}
\newcommand{\ra}{\rangle}

\newcommand{\conv}{\mathrm{conv\,}}

\begin{document}

\title{Linear programming duality for geometers}

\author{Gergely Ambrus}
\email[G. Ambrus]{ambrus@renyi.hu}

\address{\'Ecole Polytechnique F\'ed\'erale de Lausanne, Chair of Combinatorial Geometry,
Station 8, CH-1015 Lausanne, Switzerland}
\address{Alfr\'ed R\'enyi Institute of Mathematics, Hungarian Academy of Sciences, Re\'altanoda u. 13-15, 1053 Budapest, Hungary}
%\date{\today}
%\thanks{Research was supported by OTKA grants 75016 and 76099.}
%\keywords{Linear programming, duality, polar body.}
%\subjclass[2010]{90C05(primary), and 52A40 (secondary)}

\maketitle

\begin{abstract}
In this short note, we present a geometric proof for the duality theorem of linear programming. Besides being self-contained and simple, the present approach also provides a transparent way for understanding  this fundamental result.
\end{abstract}

\section{Motivation}

Since its discovery in the late 1930's, linear programming  has been arguably the most widely used method of applied mathematics. In 1947, von Neumann set one of its fundamental theoretical cornerstones by proving the Duality Theorem, which provides a duality relation between linear programs (the first published version is due to Gale, Kuhn and Tucker \cite{gkt}). Several proofs have been found to this important result, most notably, either using Dantzig's Simplex Algorithm \cite{dantzig}, or applying a relative of the Farkas lemma \cite{matousek, schrijver}. Our goal in this short note is to provide an alternative, geometric proof, which also provides a transparent interpretation.

\section{Geometric framework}

The following concepts are standard notions and results in convex geometry, and being such, they are well covered in the existing literature. A systematic treatment of them may be found, for example, in \cite{gardner} or \cite{schneider}.

The main object of study will be a {\em convex body} living in $\R^n$: a closed set $K \subset \R^n$ which is convex, that is, along with any two points $a,b \in K$, it also contains the segment connecting $a$ and $b$. For the present note, we allow $K$ to be unbounded. The most crucial property of convex sets is the {\em separation theorem} (or, the Hahn-Banach theorem), which asserts that given two disjoint convex sets, there exists a hyperplane that separates them: the two sets lie in the opposite (closed) half-spaces. In particular, it implies that each convex body may be written as an intersection of closed half-spaces.

Given a set $S \in \R^n$, its {\em convex hull}, denoted by $\conv S$, is the smallest convex set that contains $S$. It is a standard exercise that the convex hull of $S$ equals to the set of all {\em convex combinations} of points of $S$, that is,
\begin{equation}\label{convcom}
\conv S = \left\{ \sum_{i=1}^m \lambda_i s_i :\ m \geq 1, \ s_i \in S , \ \lambda_i \geq 0,  \textrm{ and } \sum_{i=1}^m \lambda_i =1 \right \}.
\end{equation}

How can we describe a convex body $K \subset \R^n$? Let us assume that $K$ contains the origin $\0$. There are two extremely natural ways to specify~$K$. First, we may stand at the origin and look around, measuring in each direction the distance to the boundary of $K$ (which may well be infinite) - note that convexity guarantees that this quantity is well defined! This method leads to the notion of the {\em radial function} of $K$, denoted by $\rho_K$, which is defined on $\R^n \setminus \{\0\}$ by
\begin{equation}\label{radialfunctiondef}
\rho_K(\uu) = \sup \{ \lambda \geq 0: \lambda \uu \in K \}.
\end{equation}

The other, equivalently natural way uses the fact that $K$ may be written as the intersection of half-spaces. Thus, this time, for each direction $\uu$, we take a hyperplane perpendicular to $\uu$ and we measure how far it has to be  shifted from $\0$ so that it becomes tangent to (supports) $K$. We arrive at the {\em support function} of $K$: it is the function $\R^n \setminus \{\0\} \rightarrow \R \cup \{\infty\}$ given by
\begin{equation}\label{suppfunctiondef}
h_K(\uu) = \sup_{\x \in K} \la \x, \uu \ra.
\end{equation}

Duality may naturally be introduced into convex geometry as follows. Let $K$ be a convex body in $\R^n$. Its {\em polar body} (or dual set), denoted by $K^*$, is given by
\begin{equation}\label{polardef}
K^* = \{ \y \in \R^n \, : \, \la \x, \y \ra \leq 1 \textrm{ for every } \x \in K \};
\end{equation}
that is, $K^*$ consists of those vectors which have a small inner product with all the vectors of $K$. This definition is familiar from functional analysis, since by the Riesz representation theorem, linear functionals on $\R^n$ can be represented as the inner product taken with a fixed vector.

One readily sees that $K^*$ is not empty, since it contains the origin. Moreover, a little thinking reveals that $K^*$ is closed; it is also not hard to check that $K^*$ is convex. Thus, polarity is indeed a relation between convex bodies. What about reflexivity? Since
\begin{equation*}
K^* = (\conv (K \cup \0))^*,
\end{equation*}
the most that we can hope for is reflexivity in the case when $\0$ is contained in $K$. This indeed turns out to be the case: the separation theorem easily implies that if $K$ is closed, convex, and $\0
\in K$, then
\begin{equation}\label{polarpolar}
(K^*)^* = K.
\end{equation}

Polarity reverses containment: if $K \subset L$ for the sets $K, L \subset \R^n$,  then $L^* \subset K^*$ holds.
Furthermore, it also interacts nicely with intersections: for $K_1, K_2$ closed, convex sets in $\R^n$,
\begin{equation}\label{intersectpolar}
(K_1 \cap K_2)^* = \conv (K_1^* \cup K_2^*).
\end{equation}

So, apart from the formal definition, what links together dual pairs of convex bodies? Since we presented two natural ways for the description of convex bodies, the ambitious may prospect a natural relationship between the support function of $K$ and the radial function of $K^*$. Voil\`{a}!

\begin{lemma}\label{suppradial}
Let $K \subset \R^n$ be a closed, convex set containing the origin. For every non-zero vector $\uu \in \R^n$,
\[
h_K(\uu) = \frac 1 {\rho_{K^*}(\uu)} \ \textrm{ and } \
h_{K^*}(\uu) = \frac 1 {\rho_K(\uu)} ,
\]
\nopagebreak
using the convention $1/\infty = 0$ and $1/0 = \infty$.
\end{lemma}

\begin{proof}
Because of the reflexivity property \eqref{polarpolar}, it suffices to prove only the first identity; but that is just the direct consequence of definitions \eqref{radialfunctiondef}, \eqref{suppfunctiondef} and~\eqref{polardef}.
\end{proof}

In a geometric language, linear programming asks for finding the support function of specific convex bodies in a given direction. These specific sets are the {\em convex polyhedrons}: they can be defined by finitely many linear inequalities, that is, they may be expressed as the intersection of finitely many closed half-spaces. Formally, a convex polytope $\R^n$ is the set of points $\x \in \R^n$ satisfying the system of linear inequalities
\[
\la \a_i, \x \ra \leq b_i
\]
for $i = 1, \ldots, m$, where $\a_i \in \R^n$ and $b_i \in \R$. For short-hand, this is written as $A \x \leq \b$, where the inequality is understood coordinate-wise, $A$ is the $m \times n$ matrix with row vectors $\a_i$, $\x \in \R^n$ a column vector, and $\b$ is the $m$-dimensional column vector with coordinates $b_i$.
A convex polyhedron need not be bounded, and we also allow the possibility for it to be empty.

Keeping our goal in mind, it is little wonder that we need to calculate the polar set of a convex polyhedron containing the origin.

\begin{lemma}\label{polarlemma}
Let $K$ be a convex polyhedron in $\R^n$ containing $\0$ (not necessarily in the interior). Assume that $K$ is  defined by the set of inequalities
\[
\la \a_i, \x \ra \leq b_i ,
\]
for $ i = 1, \ldots, n$, where $\0 \neq \a_i \in \R^n$; the condition $\0 \in K$ implies that  $b_i  \geq 0$ for all $i$.  Then $K^*$ is the convex set given by
\begin{equation}\label{polarpolyhedron}
K^* = \left\{ \sum_{i=1}^n \mu_i \mathbf{a_i}  : \mu_i \geq 0 \textrm{ for all } i=1, \ldots, n,   \textrm{ and } \sum_{i=1}^n \mu_i  b_i \leq 1 \right\}.
\end{equation}
\end{lemma}

\begin{proof}
The polar set of a half-space of the form $\{ \x: \la \a, \x \ra \leq b\}$ with $b>0$ is the segment $[0, \a/b]$, whereas the polar set of the half-space $\{ \x : \la \a, \x \ra \leq 0\}$ is the half-line $\{ \lambda \a: \lambda \geq 0\}$.
Thus, \eqref{intersectpolar} implies that
\[
K^* = \conv \{ [\0, \a_i / b_i],\ i=1, \ldots, n \},
\]
where  $[\0, \a / 0]$ denotes the half-line emanating from the origin in direction~$\a$. Switching to convex combinations, in view of \eqref{convcom}, this translates exactly to the above form.
\end{proof}

\section{Linear programs and duality}

We have already seen the geometric interpretation of a linear program. For the precise definition, we follow the standard nomenclature \cite{matousek, schrijver}. Let $A$ be a real $m \times n$ matrix, and $\c \in \R^n$, $\b \in \R^m$ be real vectors. The {\em linear program} associated to $A$, $\b$ and $\c$ is the following problem (P):

\begin{equation}
 \textit{ Maximise } \la \c , \x\ra \textit{ over } \x \in \R^n,
 \textit{ subject to } A \x \leq \b \textit{ (coordinate-wise).}
 \tag{P}\label{primalprog}
\end{equation}

The vectors satisfying the linear constraints $\la \a_i, \x \ra \leq b_i$, $i = 1, \ldots, m$, are called {\em feasible solutions}; the set of them is a convex polyhedron $F$, which may be empty or unbounded. The linear quantity to be maximised,  $\la  \c, \x \ra$, is called the {\em objective function}. Comparing to \eqref{suppfunctiondef}, the maximum sought after is $h_K(\c)$, the value of the support function of $K$ at $\c$. The solution of the program, that is, the maximum of the objective function will be denoted by $\nu_{max}$.

A linear program may not have a solution due to two possible reasons. Either $F$ is empty, in which case the program is said to be {\em infeasible}. The other possibility is that $\la \c, \x \ra$ attains arbitrarily large values over the feasible solutions, implying that $F$ is infinite; we then call the program {\em unbounded}. For the sake of convenience, in these cases we define the solution of the program to be $\nu_{max} = -\infty$ and $\nu_{max} = \infty$, respectively.

To every linear program in the form \eqref{primalprog}, we are going to assign a corresponding linear program, called the {\em (asymmetric) dual program}, defined as follows:
\begin{equation}
 \textit{Minimise } \la \b , \y\ra \textit{ over } \y \in \R^n,
 \textit{ subject to } A^\top \y = \c
\textit{ and } \y \geq \0. \tag{D}\label{dualprog}
\end{equation}

\noindent
Note that the condition $A^\top \y = \c$ may also be written as $\sum_{i=1}^m y_i \a_i = \c$.

If the solution to the dual program exists, we denote it by $\nu_{min}$. Again, \eqref{dualprog} may be infeasible or unbounded; accordingly, we set $\nu_{min} = \infty$ and $\nu_{min} = -\infty$.

Our goal is to link $\nu_{min}$ and $\nu_{max}$ together. Here comes the first result.
\begin{lemma}[Weak duality theorem]\label{weakdual}
For every feasible solution $\y$ of the dual program {\em (D)}, the value $\la \b, \y \ra$ provides an upper bound for the solution of the primal program {\em (P)}; in particular, $\nu_{max} \leq \nu_{min}$.
\end{lemma}

\begin{proof} If (P) or (D) are infeasible, then the inequality automatically holds. Otherwise, assume that $\y$ is a feasible solution of (D), and $\x$ is a feasible solution of (P). Since $A^\top \y = \c$, $A \x \leq \b$, and $\y$ is non-negative,
\begin{equation*}%\label{xyineq}
\la \c, \x \ra = \la  A^\top \y, \x \ra = \la \y, A \x \ra \leq \la  \y,\b \ra. \qedhere
\end{equation*}
\end{proof}

\noindent
In particular, the weak duality theorem implies that if \eqref{primalprog} or \eqref{dualprog} is unbounded, then the other program  must be infeasible.

\begin{theorem*}[Duality of linear programming]\label{linprogdual}
If the primal program \eqref{primalprog} is feasible and bounded, then the dual program \eqref{dualprog} is feasible and bounded as well, and the solutions of the two programs are equal: $\nu_{max} = \nu_{min}$.
\end{theorem*}

\begin{proof}
We start by noticing that by virtue of Lemma~\ref{weakdual},  \eqref{dualprog} must be bounded.
If $\c = \0$, then $\nu_{max} = 0$. Since $\y = \0$ is a feasible solution of (D), $\nu_{min} \leq 0$; on the other hand, Lemma~\ref{weakdual} guarantees that $\nu_{min} \geq 0$. Hence, $\nu_{min} = 0 = \nu_{max}$.
Thus, from now on, we suppose that $\c \neq \0$.

Let us first handle the case when $\0 \in F$, that is, the origin is a feasible solution of \eqref{primalprog}. It follows that $\b \geq \0$. Lemma~\ref{suppradial} states that
\begin{equation}\label{suppradeq}
\frac 1 {\nu_{max}} = \frac 1 {h_K(\c)} = \rho_{K^*}(\c) = \sup \{\rho \geq 0 : \rho \, \c \in K^*\}.
\end{equation}

\noindent
By Lemma~\ref{polarlemma}, the polar set of $F$ is
\begin{equation}\label{k*expr}
K^* = \left \{ \sum_{i=1}^m z_i\, \mathbf{a_i} : \z \geq \0 \textrm{ and } \ \la \b, \z \ra \leq 1 \right \},
\end{equation}
where $\z = (z_1, \ldots, z_m) \in \R^m$. Thus, we seek the supremum of the non-negative numbers $\rho$ such that the vector $\rho\, \c$ has a representation of the~form
\begin{equation}\label{rhomaxeq}
\rho \, \c = \sum_{i=1}^m z_i \a_i,
\end{equation}
where $z_i \geq 0$ for $i= 1, \ldots, m$,  and $\la \b, \z \ra \leq 1$. Since \eqref{primalprog} is bounded, \eqref{suppradeq} shows that this supremum is strictly positive. For any $\rho >0$ satisfying \eqref{rhomaxeq},  write $y_i = z_i / \rho$ for $i=1, \ldots, m$. Equation \eqref{rhomaxeq} transforms into
\[
\c = A^\top \y
\]
with $\y \geq \0$, where $\y$ satisfies that $\la \b, \y \ra \leq 1/ \rho $. We arrive exactly to the dual program \eqref{dualprog}: maximising $\rho$ is equivalent to minimising $\la \b, \y \ra$. Furthermore, $\rho_{K^*}(\c) >0 $ shows that \eqref{dualprog} is feasible. We conclude by \eqref{suppradeq}, which leads to
\[
\nu_{min} = \frac 1 {\rho_{K^*}(\c)} = h_K(\c) = \nu_{max}.
\]

Now, let us turn to the general case. By the assumption, there exists a feasible solution of \eqref{primalprog}, say $\x_0 \in F$. We translate $F$ to $F-\x_0$. Formally, for $i = 1, \ldots, m$, set $b_i' = b_i - \la \x_0, \a_i\ra$, and as usual, let $\b' = (b_1', \ldots, b_m')$. Furthermore, for any $\x \in \R^n$, let $\x' = \x - \x_0$. Define a new linear program~by
\begin{equation}
 \textit{ Maximise } \la \c , \x'\ra \textit{ over } \x' \in \R^n,
 \textit{ subject to } A \x' \leq \b' .
 \tag{P'}\label{pprime}
\end{equation}
Clearly, \eqref{pprime} is feasible and bounded (geometrically, we search for the support function of the translated body). Since $\0$ is a feasible solution of \eqref{pprime}, the previous argument applies with
\begin{equation}
 \textit{ Minimise } \la \b' , \y\ra \textit{ over } \y \in \R^n,
 \textit{ subject to } A^\top \y = \c \textrm{ and } \y \geq \0 .
 \tag{D'}\label{dprime}
\end{equation}
Since the feasible solutions of \eqref{dualprog} and \eqref{dprime} are the same, we readily conclude that \eqref{dualprog} is feasible. Finally, the identities
\[ \la \c , \x'\ra = \la \c , \x\ra - \la \c , \x_0\ra\]
and
\[
\la \b' , \y\ra = \la \b, \y \ra - \sum_{i=1}^m \la \x_0, \a_i\ra y_i = \la \b, \y \ra  - \la \x_0, A^\top \y \ra = \la \b, \y \ra - \la \x_0, \c \ra
\]
show that since the solutions of \eqref{pprime} and \eqref{dprime} are equal, the solutions of \eqref{primalprog} and \eqref{dualprog} must agree as well.
\end{proof}

We note that since there is no requirement posed for $\b$, it may well happen that (D) is feasible and bounded while (P) is infeasible, and thus, this form of LP duality is not an involution. However, one can easily fix this shortfall by adding the extra constraints $\x \geq \0$ to the primal program (resulting in the {\em canonical form}) and relaxing the constraints of the dual program to $A^\top \y \geq \c$, $\y \geq \0$ (the {\em symmetric dual form}). It follows from the duality theorem that a linear program in  canonical form is feasible and bounded if and only if its symmetric dual is so.

\end{document}